\documentclass[12pt,a4paper,reqno]{amsart}

\usepackage{amsmath,amsthm,amsfonts}

\newif\iftopbiac\topbiacfalse%\topbiactrue
\usepackage{graphicx}
\usepackage{latexsym}
\usepackage{amsfonts}

\makeindex

\newcommand{\epsfig}[1]{\reminder{Missing picture here!}}

\newtheorem{satz}{\hspace{-1.5mm}}[section]

\newcommand{\lineclear}
      {\rule{0pt}{0pt}\nopagebreak\par\nopagebreak\noindent}
\newenvironment {lemma} [1]
{\begin{satz} {\bf Lemma (#1) }\nopagebreak}{\end{satz}}
\newenvironment {defi} [1]
{\begin{satz} {\bf Definition (#1)}\nopagebreak}{\end{satz}}
\newenvironment {theo} [1]
{\begin{satz} {\bf Theorem (#1)}\nopagebreak}{\end{satz}}

\newenvironment {coro} [1]
{\begin{satz} {\bf Corollary (#1) }\nopagebreak}{\end{satz}}
\newenvironment {coro*}
{\begin{satz} {\bf Corollary }\nopagebreak}{\end{satz}}

\newenvironment {algo} [1]
{\begin{satz} {\bf Algorithm (#1) }\nopagebreak}{\end{satz}}

\newenvironment{claim}
{\par\medskip\noindent\textbf{Claim.}\begin{itshape}}{\end{itshape}\par\medskip\noindent}
\def\reminder#1{{\sf #1} }
\def\hide#1{}
\newcommand{\proofstep}[1]{\medskip{\textbf{#1}}}

\newcommand{\Nplus}{{\mathbb N^*}}

\newcommand{\Z}{\mathbb {Z}}
\newcommand{\R}{\mathbb {R}}

\newcommand\M{\mathcal M}

\newcommand{\Circle}{{\mathbb S}^1}

\newcommand\sm{\setminus}
\newcommand\ovl[1]{\overline {#1}}
\newcommand{\0}{{\tt 0}}
\newcommand{\1}{{\tt 1}}

\renewcommand{\*}{{\tt \star}}
\newcommand{\Sym}{\Sigma^1}
\newcommand{\Symg}{\Sigma}
\newcommand{\Syms}{\Sigma^\star}

\newcommand{\IntAddr}{\to}
\newcommand{\IntAdr}{\to}

\newcommand{\orb}{\mbox{\rm orb}}

\newcommand{\evil}{evil}

\renewcommand\theta{\vartheta}

\newcommand{\remark}{\par\medskip \noindent {\sc Remark. }}

\newcommand{\comment}[1]{\marginpar{#1}}

\begin{document}
\title[Bernoulli measure of admissible kneading sequences \today]
{Bernoulli measure of complex admissible kneading sequences}
\hide{ \\[4mm] Version of \today}


\author{Henk Bruin}
\author{Dierk Schleicher}
\address{Department of Mathematics, University of Surrey, Guildford GU2 7XH, 
United Kingdom}
\email{H.Bruin@surrey.ac.uk}

\address{
Jacobs University Bremen, Research I,  P.O. Box 750 561, D-28725 Bremen,
Germany}
\email{dierk@jacobs-university.de}

\begin{abstract}
Iterated quadratic polynomials give rise to a rich collection of different dynamical systems that are parametrized by a simple complex parameter $c$. The different dynamical features are encoded by the \emph{kneading sequence} which is an infinite sequence over $\{0,\1\}$. Not every such sequence actually occurs in complex dynamics. The set of admissible kneading sequences was described by Milnor and Thurston for real quadratic polynomials, and by the authors in the complex case. We prove that the set of admissible kneading sequences has positive Bernoulli measure within the set of sequences over $\{0,\1\}$. 
\end{abstract}

\maketitle

\section{Introduction}\label{sec:intro}

One of the reasons why holomorphic dynamics is a successful subject is because the rigidity of the complex structure makes it possible to describe many properties in terms of symbolic dynamics. One of the most important concepts is that of the \emph{kneading sequence}. In its simplest form, for real quadratic polynomials, it is a sequence of symbols ``left'' and ``right'' describing the location of the critical orbit with respect to the point of symmetry (the critical point). Milnor and Thurston~\cite{MiT} gave a precise criterion which possible left/right-sequences occur as kneading sequences of real quadratic polynomials, or equivalently of any unimodal real map (but they write their combinatorics in a slightly different equivalent way, keeping track of whether the map is locally increasing or decreasing). 

Kneading sequences have a natural generalization to complex quadratic polynomials, normalized as $z\mapsto z^2+c$. To see this, suppose the dynamic ray at angle $\theta$ lands at the critical value, so that the rays at angles $\theta/2$ and $(1+\theta)/2$ land at the critical point. Then the kneading sequence $\nu(\theta)=\nu_1\nu_2\nu_3\dots\in\{\0,\1\}^\Nplus$ of the angle $\theta$ can be defined as follows:
\[
\nu_k \in \{ \0 , \1 \}^{\Nplus}, \quad \nu_i = \left\{
\begin{array}{ll}
\1 & \text{ if } 2^{(i-1)}\theta  \in (\,\theta/2,(\theta+1)/2\,) \;;\\
\0 & \text{ if } 2^{(i-1)}\theta \in (\,(\theta+1)/2,\theta/2\,) \;;\\
\* & \text{ if } 2^{(i-1)}(\theta) \in\{ \theta/2,(\theta+1)/2\} \;.
\end{array} \right.
\]
Since our quadratic polynomials are monic (have leading coefficient $1$), we define dynamic rays in the canonical way so that every ray at angle $\theta$ approaches $\infty$ at angle $2\pi\theta$. Then it is easy to see that the Milnor-Thurston kneading sequence of a real quadratic polynomial with the $\theta$-ray landing at the critical value equals the kneading sequence $\nu(\theta)$ (identifying ``left'' by $\1$ and ``right'' by $\0$). But this latter definition applies to all angles $\theta\in\Circle=\R/\Z$ and thus to all complex polynomials (at least those for which some dynamic ray lands at the critical value; but this question does not matter from a combinatorial point of view).

Some technical remarks:
The entry $\*$ occurs (at position $n$) if and only $\theta$ is periodic of period $n$; this boundary case happens in the Milnor-Thurston case also for periodic critical points. And it may happen that several dynamic rays land at the critical value, defining a kneading sequence for each. It turns out that the resulting kneading sequence is independent of this choice (this is related to the fact that, at least for postcritically finite polynomials,  the critical value is always an endpoint of the connected hull of the critical orbit within the Julia set).

If the filled-in Julia set is locally connected and has no interior (so the Julia set is a dendrite), then it turns out that the kneading sequence alone allows one to give a complete topological description of the Julia set and its dynamics. It may happen that different complex quadratic polynomials have the same kneading sequence. In the dendrite case, this happens if and only if different Julia sets are topologically conjugate (by a conjugacy \emph{not} necessarily respecting the embedding into the complex plane), and this is related to certain symmetries of the Mandelbrot set that are best described in terms of \emph{internal addresses} (compare \cite{IntAddrNew}). (If the main conjecture about quadratic polynomials holds --- local connectivity of the Mandelbrot set, or equivalently, combinatorial rigidity --- then the only conjugacy that respects the embedding into the plane is the identity.)

In \cite{IntAdr} the question was raised which sequences in $\{\0,\1\}^\Nplus$ occur as kneading sequences of complex quadratic polynomials or equivalently as kneading sequences of angles; we call such kneading sequences \emph{complex admissible}. This extension of the Milnor-Thurston characterization from the real to the complex case was answered in \cite[Theorem~4.2]{KneadingAdmiss} in terms of a necessary and sufficient combinatorial condition involving internal addresses: see Definition~\ref{DefAdmissCond}.

In this note we prove the following result.
\begin{theo}{Positive Measure for Admissible Kneading Sequences}
\label{ThmAdmissMeas}\lineclear
The set of admissible kneading sequences has positive
$(\frac12$-$\frac12)$-product measure as subset of $\{ \0, \1\}^\Nplus$.
\end{theo}

In fact, the same result holds for any $(p,1-p)$-product measure with $p\in(0,1)$.

\looseness-1

The admissibility condition is so that the non-complex admissible kneading sequences form a countable union of \emph{cylinders} in $\{\0,\1\}^\Nplus$ (a $n$-cylinder is the set of sequences with common first $n$ entries). It thus suffices to discuss admissibility of periodic sequences.
\hide{, or equivalently sequences that we call \emph{$\*$-periodic}: those sequences that have period $n$ and where a single $\*$ occurs exactly at the end of the period, i.e., at positions that are multiples of $n$.}

The fundamental tool for characterizing complex admissible kneading sequences is the fact (shown in \cite{BKS}) that every periodic kneading sequence $\nu$ has an associated \emph{abstract Hubbard tree} which is a finite topological tree that satisfies all properties of Hubbard trees of quadratic polynomials, except that it may possibly not have an embedding into the plane that is respected by the dynamics. But it is a finite tree $T$ with a continuous surjective map $f\colon T\to T$ of degree at most two, with a single critical point which has a finite orbit that contains all endpoints of $T$, and subject to a certain expansitivity condition. Such a tree gives much more structure to a kneading sequence, and it can be decided relatively easily whether or not it has an embedding compatible with the dynamics, and thus whether or not the given kneading sequence is complex admissible. The only obstruction is a periodic branch point where the first return map does not permute all local branches transitively (and not even with the same period); we call such a branch point \emph{evil}. We have a precise condition if a given kneading sequence $\nu$ has an evil branch points of any period $m$: see Definition~\ref{DefAdmissCond}; if this happens, we say that $\nu$ \emph{fails the admissibility condition for period $m$}. A kneading sequence $\nu$ (periodic or not) is complex admissible if and only if it does not fail the kneading sequence for any period $m$.

The simplest $\*$-periodic kneading sequence that is not complex admissible is $\nu=\ovl{\1\0\1\1\0\*}$; the associated Hubbard is shown in Figure~\ref{Fig:Evil}. The sequence $\nu$ fails the admissibility condition for period $3$, corresponding to an evil branch point of period $3$ in the tree. The entire $6$-cylinder of sequences starting with $\1\0\1\1\0\0\dots$ fails the admissibility condition for period $m=3$.

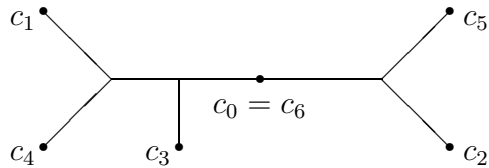
\begin{figure}[htb]
\begin{center}
\begin{minipage}{90mm}
\unitlength=9mm
\begin{picture}(10,3.2) \let\ts\textstyle
\put(3,2){\line(1,0){4}} \put(2,1){\line(1,1){1}}
\put(2,3){\line(1,-1){1}} \put(4,2){\line(0,-1){1}}
\put(7,2){\line(1,1){1}} \put(7,2){\line(1,-1){1}}

\put(5.2,2){\circle*{0.1}} \put(4.5,1.5){$c_0 = c_6$}
\put(2,3){\circle*{0.1}} \put(1.5,2.8){$c_1$}
\put(8,1){\circle*{0.1}} \put(8.2,0.8){$c_2$}
\put(4,1){\circle*{0.1}} \put(3.5,0.8){$c_3$}
\put(2,1){\circle*{0.1}} \put(1.5,0.8){$c_4$}
\put(8,3){\circle*{0.1}} \put(8.2,2.8){$c_5$}
\end{picture}
\end{minipage}
\vskip-20pt
\end{center}
\caption{The Hubbard tree with kneading sequence
$\overline{\1\0\1\1\0\*}$ contains an orbit of evil branch points of period $3$: the first return map interchanges two arms and fixes the third.
We write $c_i = f^{\circ i}(0)$.}
\label{Fig:Evil}
\end{figure}

Further properties on Hubbard trees, including the precise construction, properties and connections to external rays, quadratic laminations etc., can be found in
\cite{Orsay, MiBook, Poirier, KneadingAdmiss, BKS, Ke2}.

\section{Basic definitions and properties}\label{sec:def}

By convention, a kneading sequence starts with $\1$. 
 We say that a
sequence $\nu$ is {\em $\*$-periodic of period $n$} if $\nu =
\ovl{\nu_1 \ldots \nu_{n-1}\*}$ with $\nu_1=\1$ and
$\nu_i \in \{ \0,\1 \}$ for $1 < i < n$.
Let
\begin{eqnarray*}
\Symg &:=& \{\0,\1\}^\Nplus \,\, , \\
\Sym &:=&
\{\nu\in\Symg\colon \mbox{the first entry in $\nu$ is $\1$}\}  \,\,,\\
\Syms &:=& \Sym\cup \{\mbox{all $\*$-periodic sequences except $\ovl\*$}\} \,\,,
\end{eqnarray*}
In order to avoid silly counterexamples, $\ovl{\*}$ is not
considered to belong to $\Syms$. All sequences in $\Syms$
will be called {\em kneading sequences}.

While kneading sequences are just binary sequences, they have a ``human-readable'' recoding in terms of \emph{internal addresses}, which are strictly increasing sequences of integers starting with $1$ (and which give an elegant description of the corresponding parameters within the Mandelbrot set $\M$: taking a path in $\M$ from the main cardioid to some parameter $c$, the successive entries of the internal address give the lowest periods of the hyperbolic components that one encounters on this path, see \cite{IntAdr}). The conversion algorithm is bijective and is based on the following function.

\begin{defi}{$\rho$-Function and Internal Address} 
\label{DefRho} \lineclear
For a sequence $\nu \in\Syms$, define
\[
\rho_{\nu}:\Nplus \to \Nplus\cup\{\infty\}, \quad
\rho_\nu(n) = \inf \{ s > n: \nu_s \neq \nu_{s-n} \}.
\]
We usually write $\rho$ for $\rho_\nu$ and call
$\orb_\rho(s) = \{ \rho^{\circ i}(s) \}_{i \ge 0}$ 
the $\rho$-orbit of $s$. The case $s=1$ is the most important one;
this is the {\em internal address} of $\nu$
and we denote it as
\[
1 =  S_0 \IntAddr S_1 \IntAddr S_2  \IntAddr \dots
\]
with $S_{k+1}=\rho(S_k)$. If $\rho^{\circ k+1}(1) = \infty$, then we say that the internal
address is finite: $1 \to \rho(1) \to \ldots \to \rho^{\circ k}(1)$.
\end{defi}

\hide{
We have $\rho(n)=\infty$ if and only if the sequence $\nu$ is
periodic and its period divides $n$. If $\nu\in\Sym$ is periodic of
exact period $n$, the number $n$ may or may not appear in the internal
address (with the first $n-1$ entries in $\nu$ fixed, these two
possibilities can be realized by putting $\0$ or $\1$ at the $n$-th
position). If it does, then the internal address stops there;
otherwise, it is infinite, see  \cite[Lemma 4.3]{BKS}. %~\ref{LemExactPeriod}.
}

The map from kneading sequences in $\Sym$ to internal addresses is
injective. In fact, the algorithm of this map (originally from
\cite[Algorithm~6.2]{IntAdr}) can easily be inverted:

\begin{algo}{From Internal Address to Kneading Sequence}
\label{AlgIntAdrKneading} \lineclear
The following inductive algorithm turns internal addresses into
kneading sequences in $\Sym$: the internal address $S_0=1$ has
kneading sequence $\ovl\1$, and given the kneading sequence $\nu^k$
associated to $1\IntAdr S_1\IntAdr\ldots\IntAdr S_k$,
the kneading sequence associated to $1\IntAdr S_1\IntAdr
\ldots\IntAdr S_k\IntAdr S_{k+1}$ consists of the first
$S_{k+1}-1$ entries of $\nu^k$, then the opposite to the entry
$S_{k+1}$ in $\nu$ (switching $\0$ and $\1$), and then repeating
these $S_{k+1}$ entries periodically.
\end{algo}
\begin{proof}
The kneading sequence $\ovl \1$ has internal address $1$. If $\nu^k$
has internal address $1\IntAdr S_1\ldots\IntAdr
S_k$ and $\nu$ is the internal address of period $S_{k+1}$ as
constructed in the algorithm, then the internal address of $\nu$
clearly starts with $1\IntAdr S_1\IntAdr\ldots\IntAdr
S_k$, and $\rho_{\nu}(S_k)=S_{k+1}$, so the internal address of
$\nu$ is $1\IntAdr S_1\ldots\IntAdr S_k\IntAdr
S_{k+1}$.
\end{proof}

The $\rho$-function is of fundamental importance in the work of 
Penrose~\cite{Pe} under the name of {\em non-periodicity function}; 
the internal address is called {\em principal non-periodicity function}.

It is by means of the $\rho$-function that {\evil} periodic branch points
can be detected. This leads to the condition given 
in \cite[Definition 4.1]{KneadingAdmiss}:

\begin{defi}{The Admissibility Condition}
\label{DefAdmissCond} \lineclear
A kneading sequence $\nu\in\Syms$ {\em fails the admissibility
condition for period $m$} (which is the period of the corresponding \evil\ branch point)
if the following three conditions hold:
\begin{enumerate}
\item
the internal address of $\nu$ does not contain $m$;
\item
if $s<m$ divides $m$, then $\rho(s)\leq m$;
\item
$\rho(m)<\infty$ and if $r\in\{1,\ldots,m\}$ is congruent to $\rho(m)$ modulo $m$, then
$\orb_\rho(r)$ contains $m$.
\end{enumerate}
A kneading sequence {\em fails the admissibility condition} if it does so for some $m\ge 1$.

\noindent
An internal address {\em fails the admissibility condition} if its
associated kneading sequence does.
\end{defi}

A kneading sequence (periodic or not) is called \emph{admissible} unless it fails this condition for any period $m$.

\section{The measure of admissible kneading sequences}\label{sec:measure}

We equip 
the space $\Sym$ of all $\0$-$\1$-sequences starting with $\1$
 with
the product topology and the $(\frac 1 2$-$\frac 1 2)$-product
measure $\mu$, normalized so that $\mu(\Sym)=1$.

\begin{defi}{$\emph n$-Admissible Cylinders and Kneading Sequences}
\label{DefAdmissGeneral} \lineclear
Let $E\subset\Sym$ be the set of all sequences which satisfy the
Admissibility Condition~\ref{DefAdmissCond} for every $m$ (the {\em
set of admissible kneading sequences}). For
$n\geq 1$, a finite word $\nu_1\dots\nu_n$ with $\nu_i\in\{\0,\1\}$ and $\nu_1=\1$
is called an {\em admissible $n$-word} if there is a $\nu\in E$ which begins with
$\nu_1\dots\nu_n$.
An {\em admissible $n$-cylinder}
is an $n$-cylinder that contains an admissible sequence.
Finally, let $E_n$ be the union of the admissible $n$-cylinders;
this is the {\em set of $n$-admissible kneading sequences}.
\end{defi}

Whether or not
$\nu\in\Sym$ fails the admissibility condition for period $m$
depends only on the first $\rho_\nu(m)$ entries in $\nu$, so if
$\nu$ fails the condition for period $m$, then an entire
$\rho_\nu(m)$-cylinder is non-admissible. Thus $E=\cap_{n\geq 1}
E_n$ is a decreasing intersection of sets which are simultaneously
open and compact.

%\newpage

\begin{lemma}{Admissible Kneading Sequences Form Cantor Set}
\label{LemAdmissCantor} \lineclear
The set $E\subset\Sym$ of admissible kneading sequences forms a
Cantor set so that $\Sym\sm E $ is dense in $\Sym$. In particular, $E$ is measurable with respect to $\mu$.
\end{lemma}
\begin{proof}
Any violation of the admissibility condition for period $m$
discards an entire cylinder subset of $\Sym$.
The set of non-admissible kneading sequences is a
union of such cylinder sets and hence open. As a closed subset of
the compact set $\Sym$, the set $E$ is compact.
Clearly, $E$ is totally disconnected because $\Sym\supset E$ is.

Recall from Algorithm~\ref{AlgIntAdrKneading} that if $S_k=\rho_\nu^{\circ k}(1)$ is an element of the internal address of $\nu$ then $\nu^k$ is the $S_k$-periodic kneading sequence that coincides with $\nu$ for $S_k$ entries.
If $\nu \in E$, then $\nu^k \in E$.
We can always extend admissible kneading sequences via direct bifurcations,
so for any admissible internal address $1\IntAdr\dots\IntAdr S_k$ and any $p > 1$, the internal address $1 \IntAdr S_1\IntAdr \ldots \IntAdr S_k \IntAdr pS_k$ is admissible. Since $k$ and $p$ are arbitrary, $\nu$ is not an isolated point.

Finally, non-admissible sequences are dense in $\Sym$ because any
finite word can be continued into a non-admissible one.
Indeed, given any $\nu=\nu_1\nu_2\nu_3\dots\in\Sym$ and any prime $m>\rho_{\nu}(1)$, define a sequence $\nu'=\nu'_1\nu'_2\dots$ with $\nu'_k=\nu_k$ unless $m$ divides $k$, choosing $\nu'_m$ so that $m$ is not in the internal address of $\nu'$, and $\nu'_{2m}\neq\nu'_m$. Then $\rho_{\nu'}(m)=2m$ and $\nu'$ fails the admissibility condition for period $m$: the  first and third conditions are clearly satisfied, and the second one is void because $m$ is prime.
\end{proof}

\begin{coro}{Admissible Periodic Sequences}
\label{CorAdmissCylinders} \lineclear
Let $R_k$ be the number of admissible periodic kneading sequences
of period $k$ in $\Sym$. Then
\[
\lim_{k\to\infty} \frac{R_k}{2^{k-1}} = \mu(E)  \,\,.
\]
\end{coro}
\begin{proof}
We have $\mu(E_k)=R_k 2^{-(k-1)}$
\hide{(Lemma~\ref{LemPeriodicInAdmissible}) }
and $\mu(E)=\lim_{k\to\infty}\mu(E_k)$.
\end{proof}

\hide{
\begin{theo}{Admissible Kneading Sequences Have Positive Measure}
\label{ThmAdmissMeas}\lineclear
The set of admissible kneading sequences has positive measure:
$\mu(E) > 0$.
\end{theo}
}

Now we are ready to prove Theorem~\ref{ThmAdmissMeas}.

\begin{proof}[Proof of Theorem~\ref{ThmAdmissMeas}]
Fix $n_1 := 100$ and define a sequence of
integers $\{ n_i \}_{i\geq 1}$ so that $n_{i+1}$ is the largest
integer less than $\frac32 n_i$ for which $n_{i+1}-n_i$ is
divisible by $2i$ (this gives $100$, $148$, $220$, $328$,
$488$,\dots).
Let $m_{i} :=({n_{i+1}-n_i})/{2}$. Then one can
check that $(1.45)^i < m_i$. Clearly, $n_{i+1}<\frac 3 2 n_i
<n_i+\frac 3 4 n_{i-1}<n_i+n_{i-1}$ and
$m_i < n_i/4$.

Let $F_1$ be some cylinder of length $n_1$ containing an
admissible sequence so that no block of $40$ consecutive $\0's$
appears among the first $n_1$ entries. For $i \geq 1$, define
\begin{eqnarray*}
F_{i+1} := \left\{
\begin{minipage}[c]{18mm}
$\nu \in F_i \,\colon$ \hfill \rule{0pt}{1pt} \\
\hfill \rule{0pt}{1pt} \\
\hfill \rule{0pt}{1pt} \\
\hfill \rule{0pt}{1pt}
\end{minipage}
\begin{minipage}[c]{87mm}
every $k \leq n_{i}+m_{i}$ has an $s \in\orb_{\rho_\nu}(k)\cap
\orb_{\rho_\nu}(1)$ \hfill\\
\,\,\,\, with $n_{i}+m_{i} < s \leq n_{i+1}$  \\
and the entries $n_i$, $n_i+1$, \dots, $n_{i+1}-1$ in $\nu$ do \\
not contain a block of $\lfloor n_i/8\rfloor$ zeroes.
\end{minipage}
\right\} \,\,
\end{eqnarray*}
For every $\nu\in F_{i+1}$, every $k\leq n_i+m_i$ satisfies
$\rho_\nu(k)\leq n_{i+1}$, so every $F_{i+1}$ is a union of
cylinder sets of length $n_{i+1}$. 

\begin{claim}
The second condition
for $F_{i+1}$ implies that for all $N\le n_{i+1}$, the first $N$ entries in any $\nu\in
F_{i+1}$ do not contain a contiguous block of $\lfloor N/4\rfloor$
zeroes. 
\end{claim}%
Indeed, for any $n_j <N$, the longest block of consecutive zeroes among entries $n_j$ \dots $n_{j+1}-1$ has length less than $n_j/8$; if these are near the end, they can be continued by less than $n_{j+1}/8$ further zeroes, yielding a total number of consecutive zeroes of less than $n_j/8+n_{j+1}/8< (5/16)n_j$, and they end at entry number $(9/8)n_{j+1}< (27/16)n_j$, so among the first $(27/16)n_j$ entries there are no more than $(5/16)n_j<(27/16)n_j/4$ consecutive zeroes. 

The two main steps in the proof are $E \supset \cap_i F_i$ and
$\mu(F_{i+1})\geq c_i \mu(F_i)$ for numbers $c_i>0$ with $\prod_i
c_i>0$, from which the conclusion will follow.

\proofstep{(1) Admissibility of cylinders $F_i$ and $E\supset\bigcap_i F_i$.}
By induction over $i\geq 1$, we show that all $\nu\in \cap_j F_j$
are $n_{i}$-admissible for all $i$. They are $n_1$-admissible by
hypothesis. Suppose $\nu\in \cap_j F_j$ is $n_i$-admissible but not
$n_{i+1}$-admissible; then $\nu$ fails the Admissibility
Condition~\ref{DefAdmissCond} for some period $m$ with
$n_i<\rho(m)\leq n_{i+1}$. Let
$r\in\{1,2,\dots,m\}$ be congruent to $\rho(m)$ modulo $m$.

If $m\leq n_{i-1}+m_{i-1}$, then by definition of $F_{i-1}$, there is
an $s\in\orb_\rho(m)\cap\orb_\rho(1)$ with $s\leq n_i$; since
$m\notin\orb_\rho(1)$, we have $\rho(m)\leq s\leq n_i$, which is a
contradiction. Hence $m > n_{i-1}+m_{i-1}$.

If $r\geq n_{i-1}+m_{i-1}$, then $\rho(m)\geq m+r >
2n_{i-1}+2m_{i-1}=n_i+n_{i-1}> n_{i+1}$, again a contradiction.
Thus $r<n_{i-1}+m_{i-1}$, and there is an $s\in\orb_\rho(r)\cap
\orb_\rho(1)$ with $s\leq n_i$. By the admissibility condition,
$m\in\orb_\rho(r)$ and $m\notin\orb_\rho(1)$, hence $s>m$.
Thus $s\in\orb_\rho(m)$ and $\rho(m)\leq s\leq n_i$ in
contradiction to $n_i$-admissibility of $\nu$. This final
contradiction shows that $E \supset \cap_i F_i$ as claimed.

\proofstep{(2) A bound on the number of jumps of $\rho$.}
Given $\nu$ and an integer $r$, let $J_\nu(r) = \{ k \leq r : \rho(k)
>    r\}$. Obviously, $J_\nu(r)$ depends only on the first $r$ entries
of $\nu$. 
\begin{claim}
For each $i\geq 1$,
\begin{equation}\label{EqJ}
\mbox{if $ \nu \in F_i$ and $n_i + m_i \leq r < n_{i+1}$
then $\# J_\nu(r) < i/2+150$\,\,.}
\tag{\protect$\*$}
\end{equation}
\end{claim}%
Indeed, if $k \in J_\nu(r)$, then $\nu_{k+1} \dots \nu_r = \nu_1
\dots \nu_{r-k}$. Suppose that $r-k> n_2$. Then there is a unique
$j\geq 1$ with $n_{j+2}\geq r-k>n_{j+1}$, and the hypothesis $r<n_{i+1}$ implies
$j<i$. Since all $k'\in\{1,2,\dots, n_j+m_j\}$ have $\rho(k')\leq
n_{j+1} <r-k$ by definition of $F_{j+1}$, it follows that all
$k'\in\{k+1,k+2,\dots,k+n_j+m_j\}$ have $\rho(k')\leq
k+n_{j+1}<r$, hence $k'\notin J_\nu(r)$.

We turn this into an inductive argument:
write $J_\nu(r)=\{k_1,k_2,k_3,\dots,k_S\}$ with $k_s<k_{s+1}$ for
all $s$. Let $j_s$ be so that $n_{j_s+2}\geq r-k_s>n_{j_s+1}$.
Then $k_{s+1}>k_s+n_{j_s}+m_{j_s}$, hence
\begin{eqnarray*}
n_{j_{(s+1)}+1} <
r-k_{s+1}
&<& r-k_s-n_{j_s}-m_{j_s}
\leq
n_{j_s+2}-n_{j_s}-m_{j_s} \\
&=&
n_{j_s}+2m_{j_s}+2m_{j_s+1}-n_{j_s}-m_{j_s}
= m_{j_s}+2m_{j_{s}+1} \\
&<& n_{j_s}/4 + n_{j_s+1}/2 < n_{j_s} \,\,,
\end{eqnarray*}
so $j_{s+1}+1< j_s$ or $j_{s+1}\leq j_s-2$. Since we started with $r-k_1<r<n_{i+1}$, after $s$ turns of this argument we have $r-k_{s+1}< n_{j_{s}}$ with $j_s\le i+1-2s$, and for $s\le (i-1)/2$ we have $j_s\le 2$, so $r-k_{s+1}<n_2$ or $k_{s+1}>r-n_2$, and there are at most $n_2$ such values of $k_s$. ``
It follows that $\# J_\nu(r) \leq
(i-1)/2+1+n_2<i/2+150$ as claimed.

\proofstep{(3) Separate treatment of sub-blocks.}
Now we begin the proof that $\mu(F_{i+1})$ is not too small compared to
$\mu(F_i)$. Suppose $i\geq 40$, which implies
$\lfloor\log_2(i/2+150)\rfloor+1\leq \lfloor i/4\rfloor$.
Let $C$ be any $n_i+m_i$-cylinder in $F_i$ and pick $\nu\in C$.
Divide the integer interval $[n_i+m_i+1, n_{i+1}]$
into $m_i/i$ blocks $B_1,B_2,\dots$ of length $i$.
By Equation~(\ref{EqJ}), $\# J_\nu(n_i+m_i) < i/2+150$, and there
are $2^i$  different possibilities for $B_1$ which extend $C$. We
claim that at least $2^{i/2}$ of them are {\em barriers} in the
sense that for all $k\in J_\nu(n_i+m_i)$, we have
$n_i+m_i+2\lfloor i/4\rfloor\in\orb_\rho(k)$.
This is useful in view of the first condition of the definition
of $F_i$.
To construct these barriers, divide $B_1$ into three
subblocks: the first two of length $\lfloor i/4\rfloor$ each, the
last of length $i-2\lfloor i/4\rfloor\geq i/2$.

In the second subblock, every entry is $\0$. The third subblock
is filled arbitrarily with $\0$ and $\1$, for which there are at
least $2^{i/2}$ possibilities. The first subblock needs more care.

In the first subblock of $B_1$, choose the first entry (at
position $n_i+m_i+1$) so that $n_i+m_i+1\in\orb_\rho(k)$ for at
least half of the elements $k\in J_\nu(n_i+m_i)$; choose the
second entry so that $n_i+m_i+2\in\orb_\rho(k)$ for at least half
of the remaining elements in $J_\nu(n_i+m_i)$, and so on. The
first $\lfloor\log_2(i/2+150)\rfloor+1$ entries in $B_1$ suffice
so that for every $k\in J_\nu(n_i+m_i)$, $\orb_\rho(k)$ contains
at least one of these positions. Since
$\lfloor\log_2(i/2+150)\rfloor+1\leq\lfloor i/4\rfloor$,
we have used only positions within the first subblock of $B_1$.
Fill the remaining entries within the first subblock of $B_1$
arbitrarily.

Since the first $N$ entries of $\nu$ do not contain a contiguous
block of $\lfloor N/4\rfloor$ zeroes, it follows that the orbit
of every $k\in J_\nu(n_i+m_i)$ visits an entry in the second
subblock of $B_1$, and $\orb_\rho(k)$ contains the last position
in the second subblock of $B_1$, which is
$n_i+m_i+2\lfloor i/4\rfloor\in\orb_\rho(k)$.

An analogous construction can be repeated for all the other blocks
$B_2$, $B_3$,\dots.
Suppose that at least a single block $B_j$ has a barrier as above.
Then the construction yields $n_{i+1}$-cylinders which satisfy the
first condition for $F_{i+1}$: for $k\leq n_{i}+m_{i}$, we
either have $\rho(k)\leq n_{i}+m_{i}$ and we consider
$\rho(k)$ instead of $k$, or $k\in J_\nu(n_{i}+m_{i}+(j-1)i)$ and
$\orb_\nu(k)$ visits $n_{i}+m_{i}+2\lfloor i/4\rfloor+(j-1)i$.
Therefore $n_{i}+m_{i}+2\lfloor i/4\rfloor+(j-1)i\in\orb_\nu(k)$
for every $k\leq n_{i}+m_{i}+2\lfloor i/4\rfloor$,
in particular for $k = 1$.

\proofstep{(4) Estimating the relative loss of measure.}
Given the $n_i+m_i$-cylinder $C$, there are $2^i$ continuations
into $n_i+m_i+i$-cylinders in $F_{i+1}$, and at least $2^{i/2}$
of them have a barrier in $B_1$ as constructed above, so at most
a relative proportion of $1-2^{-i/2}$ has no barrier within
$B_1$. The same bound for the relative proportions holds for
$B_2$, $B_3$, \dots. We find that the proportion of
$n_{i+1}$-cylinders in $C$ with no barrier at any block $B_j$ for
$j = 1, \dots, m_i/i$ is $(1-2^{-i/2})^{m_i/i}$.
We have $m_i > (1.45)^i > 2^{i/2}$, so
$m_i2^{-i/2}>\alpha^i$ with $\alpha>1$. Therefore
\[
(1-2^{-i/2})^{m_i/i}
< %&<&
\left(\exp(-2^{-i/2})\right)^{m_i/i}
= \exp\left(-2^{-i/2}\frac{m_i}{i}\right)
< %&<&
\exp(-\alpha^i/i)
< i/\alpha^i,
\]
and the relative proportion of $n_{i+1}$-cylinders in $F_{i+1}$
satisfying the first condition is at least $1-i/\alpha^i$.

We still have to take care of the second condition for $F_{i+1}$.
Among the $2^{2m_i}$ continuations from any $n_i$-cylinder into
an $n_{i+1}$-cylinder, there are $2^{2m_i-\lfloor n_i/8\rfloor}$
which have $\lfloor n_i/8\rfloor$ contiguous entries $\0$
beginning at any given position
$n_i,n_i+1,\dots,n_{i+1}-\lfloor n_i/8\rfloor$, and less than
$2m_i2^{2m_i-\lfloor n_i/8\rfloor}$ which have $\lfloor
n_i/8\rfloor$ contiguous $\0$'s between entries $n_i$ and
$n_{i+1}-1$ (inclusively). The proportion of $n_{i+1}$-cylinders
in $F_i$ which do not contain $\lfloor n_i/8\rfloor$ consecutive
$\0$'s is thus at least $1-2m_i2^{-\lfloor n_i/8\rfloor}>
1-\frac12 n_i2^{-\lfloor n_i/8\rfloor}$.

Therefore
\begin{eqnarray*}
\mu(F_{i+1})
&\geq&
\left[1-\left(1-2^{-i/2}\right)^{{m_i}/i}\right]
\left[1-\frac12 n_i2^{-\lfloor n_i/8\rfloor}\right]\mu(F_i) \\
&>&
\left[1-\frac{i}{\alpha^i}\right]
\left[1-\frac{n_i}{2\cdot2^{\lfloor n_i/8\rfloor}}\right]\mu(F_i)
\end{eqnarray*}
for $i \geq 40$. Since $\sum_i i/\alpha^i<\infty$ and $\sum_i
n_i/2^{\lfloor n_i/8\rfloor}<\infty$, it follows
\[
\mu(E) \geq \mu \left(\bigcap_i F_i\right)
\geq \mu(F_{40}) \prod_{i \geq 40}
\left[1-\frac{i}{\alpha^i}\right]
\left[1-\frac{n_i}{2\cdot2^{\lfloor n_i/8\rfloor}}\right] > 0
\]
because $F_{40}$ contains at least one cylinder set. This proves
the theorem.
\end{proof}

\remark{A similar proof shows that $\mu_p(E) > 0$ for
the $(p,1-p)$-product measure $\mu_p$.}

\begin{table}[htbp]
{\tiny
\[
\begin{array}{r@{\hskip 1cm}r@{\hskip 1cm}r@{\hskip 1cm}r}
     n & \mbox{ non-admis.} &
\mbox{ total} &\mbox{ non-admis./total} \\
\hline
     6 &               1 &            32 &  0.031250 \\
     7 &               2 &            64 &  0.031250 \\
     8 &               7 &           128 &  0.054688 \\
     9 &              17 &           256 &  0.066406 \\
     10&              44 &           512 &  0.085938 \\
     11&              96 &          1\,024 &  0.093750 \\
     12&             221 &          2\,048 &  0.107910 \\
     13&             473 &          4\,096 &  0.115479 \\
     14&            1\,028 &          8\,192 &  0.125488 \\
%\hide{
     15&            2\,160 &         16\,384 &  0.131836 \\
     16&            4\,544 &         32\,768 &  0.138672 \\
     17&            9\,408 &         65\,536 &  0.143555 \\
     18&           19\,488 &        131\,072 &  0.148682 \\
     19&           39\,984 &        262\,144 &  0.152527 \\
     20&           81\,963 &        524\,288 &  0.156332 \\
     21&          167\,138 &       1\,048\,576 &  0.159395 \\
     22&          340\,393 &       2\,097\,152 &  0.162312 \\
     23&          691\,104 &       4\,194\,304 &  0.164772 \\
     24&         1\,401\,610 &       8\,388\,608 &  0.167085 \\
%}
\hide{
\vdots\hskip .1cm& \vdots\quad& \vdots\quad& \vdots\hskip .6cm \\
}
     25&         2\,836\,989 &      16\,777\,216 &  0.169098 \\
     26&         5\,737\,023 &      33\,554\,432 &  0.170977 \\
     27&         11\,586\,451 &      67\,108\,864 &  0.172652 \\
     28&         23\,382\,085 &     134\,217\,728 &  0.174210 \\
     29&         47\,143\,911 &     268\,43\,5456 &  0.175625 \\
     30&         94\,995\,724 &     536\,870\,912 &  0.176943 
%\\
%     31&        191\,292\,387 &    1\,073\,741\,824 &  0.178155 \\
%     32&        385\,016\,428 &    2\,147\,483\,648 &  0.179287
\end{array}
\]
}
\caption{The number of non-admissible cylinders of length $n$,
among all the $2^{n-1}$ cylinders in $\Sym$ of this length, for
$n=6,7,\ldots,30$.
}
\label{TabAdmissible}
\end{table}

\par \medskip \noindent
A computation (shown in Table~\ref{TabAdmissible}) suggests
that about $80$ percent of $\Sym$ is admissible. Unfortunately,
the convergence is too slow to make precise  estimates for $\mu(E)$.
\hide{\comment{The inequalities in the proof should make more precise estimates possible, but probably only for much larger $n$?}
}
Let us discuss some examples from this table: for $n=6$,
the $2^5$ total cylinders
of length $6$ are the binary words of length $6$ starting with
$\1$, and the only non-admissible cylinder is $\1\0\1\1\0\0$ (the
only non-admissible $\*$-periodic sequence is $\ovl{\1\0\1\1\0\*}$,
but the cylinder $\1\0\1\,\,\1\0\1$ is admissible). For $n=7$, the
two non-admissible cylinders are the two continuations of
$\1\0\1\1\0\0$; for $n=8$, we have four continuations, as well as
the new cylinders $\1\0\0\1\,\1\0\0\0$, $\1\1\0\1\,\1\1\0\0$ and
$\1\0\1\1\,\1\1\0\0$.
\hide{In the sense of
Proposition~\protect\ref{PropBranchNonAdmiss}, these sequences
branch off from the admissible subtree at the kneading sequences
$\ovl{\1\0\0\*}$, $\ovl{\1\1\0\*}$ and $\ovl{\1\0\1\1\*}$ of
periods $4$, $4$ and $5$(!).
}


\begin{thebibliography}{MM9}


\bibitem[BS]{KneadingAdmiss} Henk Bruin, Dierk Schleicher,
{\em Admissibility of kneading sequences and structure of Hubbard trees for quadratic polynomials,}
J.\ London.\ Math.\ Soc.\ {\bf 8} (2009),  502--522.

\bibitem[BKS]{BKS} Henk Bruin, Alexandra Kaffl, Dierk Schleicher,
{\em  Existence of quadratic Hubbard trees,}
Fund.\ Math.\ {\bf 202} (2009), 251--279.

\bibitem[DH]{Orsay} Adrien~Douady, John~Hubbard,
{\em \'Etudes dynamique des polyn\^omes complexes
I \& II},
Publ. Math. Orsay. (1984-85) {\em(The Orsay notes)}.

\bibitem[K]{Ke2} Karsten~Keller,
{\em Invariant factors, Julia equivalences and the (abstract)
Mandelbrot set,} Springer Lect. Notes Math. {\bf 1732} (2000).

\bibitem[LS]{IntAdr} Eike~Lau, Dierk~Schleicher,
{\em Internal addresses in the Mandelbrot set and irreducibility
of polynomials,}
Stony Brook Preprint {\bf \#19} (1994).

\bibitem[M]{MiBook} John~Milnor,
{\em Dynamics in one complex variable. Introductory lectures}.
Third edition. Annals of Mathematical Study \textbf{160}. Princeton University Press 2006.

\bibitem[MT]{MiT} John Milnor, William Thurston,
{\em On iterated maps of the interval,}
LNM {\bf 1342} (J.\ C.\ Alexander ed.). 465--563 (1988).


\bibitem[Pe]{Pe} Chris~Penrose,
{\em On quotients of shifts associated with dendrite Julia sets
of quadratic polynomials,}
Ph.D.\ Thesis, University of Coventry, (1994).


\bibitem[Po]{Poirier} Alfredo~Poirier,
\emph{Hubbard trees}. 
Fundamenta Mathematicae \textbf{208} 2 (2010), 193--248. 

\bibitem[S]{IntAddrNew} Dierk Schleicher,
\emph{Internal addresses in the Mandelbrot set and Galois groups of polynomials}; arXiv:math/9411238. 

\bibitem[T]{thurston} William Thurston,
{\em On the geometry and dynamics of iterated rational maps}.
In: Dierk Schlei\-cher (ed.), \emph{Complex dynamics: Families and friends.} A K Peters, Wellesley, 2009, pp. 3--109.


\end{thebibliography}
\end{document}